\documentclass{gtpart}

\title{Finite knot surgeries and Heegaard Floer homology}

\author[M. Doig]{Margaret I. Doig}
\givenname{Margaret I.}
\surname{Doig}
\address{Department of Mathematics\\
Syracuse University\\
215 Carnegie Building\\
Syracuse, NY 13244-1150}
\email{midoig@syr.edu}
 
\arxivreference{arXiv:1201.4187}
\arxivpassword{}

\volumenumber{}
\issuenumber{}
\publicationyear{}
\papernumber{}
\startpage{}
\endpage{}
\doi{}
\MR{}
\Zbl{}
\received{}
\revised{}
\accepted{}
\published{}
\publishedonline{}
\proposed{}
\seconded{}
\corresponding{}
\editor{}
\version{}

\usepackage{amssymb}
\usepackage{amsmath}
\usepackage{amsthm}
\usepackage[all]{xy}
\usepackage{rotating}
\usepackage{graphicx}

\newtheorem{thm}{Theorem}

\newtheorem{cor}[thm]{Corollary}
\newtheorem{conj}[thm]{Conjecture}
\theoremstyle{definition}
\newtheorem{ex}[thm]{Example}

\newcommand{\surg}{S^3_{p/q}}
\newcommand{\Zed}{\mathbb{Z}}
\newcommand{\Tnaught}{\mathcal{T}^+_0}
\newcommand{\spinc}{\mathrm{Spin^c}}
\newcommand{\spincstr}{\mathfrak{t}}
\newcommand{\spincstrextend}{\mathfrak{s}}
\newcommand{\Char}{\mathrm{Char}}
\newcommand{\weight}{\mathrm{m}}

\newcommand{\hf}{\widehat{HF}}

\begin{document}

\begin{abstract}
It is well known that any 3-manifold can be obtained by Dehn surgery on a link but not which ones can be obtained from a knot or which knots can produce them. We investigate these two questions for elliptic Seifert fibered spaces (other than lens spaces) using the Heegaard Floer \emph{correction terms} or \emph{d-invariants} associated to a 3-manifold $Y$ and its torsion $\spinc$ structures. For $\pi_1(Y)$ finite and $|H_1(Y)|\leq 4$, we classify the manifolds which are knot surgery and the knot surgeries which give them; for $|H_1(Y)|\leq 32$, we classify the manifolds which are surgery and place restrictions on the surgeries which may give them.
\end{abstract}

\maketitle

\section{Introduction}\label{sectintro}

In the 1960s, Wallace \cite{wallace} and Lickorish \cite{lickorish} showed that any oriented 3-manifold can be constructed by \emph{Dehn surgery}\footnote{To perform \emph{p/q-[Dehn] surgery} on a knot $K$ embedded in $S^3$, remove an open neighborhood $N(K)$ homeomorphic to a solid torus and replace it by identifying a meridian of the solid torus with $p\mu + q\lambda$ in the knot complement. Here, $\mu$ and $\lambda$ are oriented curves on $\partial \overline{N(K)}$ where $\mu$ bounds a disk in $\overline{N(K)}$; $\lambda$ is null-homologous in $H_1(S^3-N(K))$; and the geometric intersection number of $\mu$ and $\lambda$ is $+1$.}  on a link in $S^3$. Soon after, Moser asked which manifolds can be constructed by surgery on a knot \cite{moser}. One may also ask which knots give each manifold. We begin to answer these two questions for elliptic (or spherical) manifolds other than lens spaces, that is, those with finite but non-cyclic fundamental group.

We know that $S^3$ only comes from trivial surgeries~\cite{gordonlueckecomplements}, and $S^1\times S^2$ arises only from 0-surgery on the unknot \cite{gabaifoliations}. On the other hand, lens spaces can come from torus knots \cite{moser} but may also arise from integral surgery on some hyperbolic knots \cite{culetal}. Berge \cite{bergeconj} proposed a comprehensive list of such surgeries using \emph{primitive/primitive} knots, which is now referred to as the Berge Conjecture and is listed as Problem 1.78~\cite{kirby}. Ozsv\'ath and Szab\'o \cite{ozszlens} gave a necessary condition on the Alexander polynomial of a knot with a lens space surgery and verified Berge's list up to $p\leq 1500$, and Greene \cite{greeneberge} verified that any lens space which is surgery on a non-trivial knot is achieved by some knot on the list (he did not verify that all knots giving lens spaces are on the list). Dean \cite{dean} proposed an extension of these results from lens spaces to small Seifert fibered spaces. However, Dean's list is not exhaustive: other hyperbolic surgeries also produce small Seifert fibered manifolds, but all such known manifolds are also given by knots from Dean's list \cite{dmm, mmm}. 

We address a subset of this case, elliptic (or spherical) manifolds other than the lens spaces, i.e., Seifert fibered manifolds with finite but non-cyclic fundamental group. If a surgery gives such a group, we will call it a \emph{finite} and \emph{non-cyclic} surgery.

The finite surgeries on torus knots are easy to identify based on Moser's classification~\cite[Propositions~3.1,~3.2,~4]{moser}. Bleiler and Hodgson explicitly listed the finite surgeries on iterated torus knots~\cite[Theorem~7]{bleilerhodgson2} based on Gordon's classification~\cite[Theorem~7.5]{gordonsatellite}; all of the resulting manifolds are also torus knot surgeries. Boyer and Zhang proved that no other satellite knots have finite surgeries~\cite[Corollary~1.4]{boyerzhang}.
 
Boyer and Zhang showed that all finite surgeries on hyperbolic knots are integral or half-integral, although it is conjectured that they are integral (see, e.g.,~\cite[Problem 177, Conjecture A]{kirby}). Additionally, any hyperbolic knot has at most five finite or cyclic surgeries, with at most one non-integral. Any two such surgeries on the same knot have distance\footnote{A surgery coefficient $p/q$ corresponds to a homology class $p\mu+q\lambda$ on $\partial \overline{N(K)}$. The \emph{distance} between two surgery coefficients is the minimum geometric intersection number of two curves representing the corresponding homology classes.} at most 3, and the distance 3 is realized by at most one pair~\cite[Theorems~1.1,~1.2]{boyerzhangii}. 

There are a variety of examples of finite surgeries on hyperbolic knots. Fintushel and Stern~\cite{fintushelsternlens} and Bleiler and Hodgson~\cite{bleilerhodgson2} commented respectively that $17$-surgery on $(-2,3,7)$ pretzel knot and $22$- and $23$-surgery on the $(-2,3,9)$ pretzel knot are finite (although all three resulting manifolds are also torus knot surgeries), and Mattman et al. showed that there are no other finite surgeries on pretzel knots ~\cite[Theorem~1.2]{mattmanpretzelfinite}, \cite[Theorem~1]{mattmanetalpretzel}. It is an interesting question for which $p$ there are finite $p/q$-surgeries on hyperbolic knots. As Zhang stated in Conjecture~$\hat{I}$ ~\cite{zhangpropi} and Kirby formulated in a remark after Problem~3.6(D) \cite{kirby}, the Poincar\'e homology sphere (the only manifold with finite $\pi_1Y$ and $|H_1(Y)|=1$) has a unique surgery construction. Ghiggini proved:

\begin{thm}\cite[Corollary 1.7]{ghigginigenusonefibered}\label{thmz1}
The Poincar\'e homology sphere is $-1$-surgery on the left-handed trefoil (or, reversing orientation, $+1$-surgery on the right-handed trefoil)
\[\Sigma(2,3,5) = S^3_{-1}(T_{3,-2})\]
and no other surgery on any knot.\footnote{Throughout this paper, we suppress the choice of orientations; unless otherwise stated, surgery coefficients are positive and Seifert fibered descriptions are the canonical ones described in Theorem~\ref{thmh1order}.}
\end{thm}

Elliptic spaces fall into a group of manifolds called \emph{L-spaces} whose Heegaard Floer homology is particularly simple (Cf \cite[Proposition~2.3]{ozszlens}). If an L-space is given by $p/q$  surgery on a knot $K$ in $S^3$, then it obeys the inequality $p/q\geq 2g(K)-1$, and the knot is fibered with one of a very small set of Alexander polynomials. The \emph{correction terms} or \emph{d-invariants} $d(Y,\spincstr)$ take a very nice form for L-space surgeries, and they can be compared to the $d(Y,\spincstr)$ calculated directly from a plumbing graph. We prove:

\begin{thm} \label{thmz2}
Up to orientation, the only finite, non-cyclic surgeries with \mbox{$p \leq 9$} are: 
\[\begin{array}{rcl}
S^3_1(T_{3,2}) &=& \left(-1; \frac{1}{2}, \frac{1}{3}, \frac{1}{5}\right) \\
S^3_2(T_{3,2}) &=& \left(-1; \frac{1}{2}, \frac{1}{3}, \frac{1}{4}\right) \\
S^3_3(T_{3,2}) &=& \left(-1; \frac{1}{2}, \frac{1}{3}, \frac{1}{3}\right) \\
S^3_4(T_{3,2}) &=& \left(-1; \frac{1}{2}, \frac{1}{2}, \frac{1}{3}\right) \\
S^3_{7/2}(T_{3,2}) ~=~ S^3_7(T_{5,2}) &=& \left(-1; \frac{1}{2}, \frac{1}{3}, \frac{2}{5}\right) \\
-S^3_8(T_{3,2}) &=& \left(-1; \frac{1}{2}, \frac{1}{2}, \frac{2}{3}\right) \\
S^3_8(T_{5,2}) &=& \left(-1; \frac{1}{2}, \frac{1}{2}, \frac{2}{5}\right) \\
S^3_{9/2}(T_{3,2}) ~=~- S^3_9(T_{3,2}) &=& \left(-1; \frac{1}{2}, \frac{1}{3}, \frac{2}{3}\right) \\
\end{array}\] 
With the possible exception of $S^3_7(T_{5,2})$ and $S^3_8(T_{5,2})$, there are no other surgeries (up to orientation) giving these manifolds. 

The following manifolds cannot be realized as any knot surgery:
\[\begin{array}{rl}
\left(-1; \frac{1}{2}, \frac{1}{2}, \frac{1}{n}\right) & if~ n\neq 3 \phantom{,5}\\
\left(-1; \frac{1}{2}, \frac{1}{2}, \frac{2}{n}\right) & if~ n\neq 3~or~5
\end{array}\]
\end{thm}

Note that there are no elliptic Seifert fibered spaces with $|H_1(Y)| =5$ or $6$; there are unique spaces for each of $|H_1(Y)|=1, 2, 3, 7,$ and $9$; and there are infinite families for both $|H_1(Y)|=4$ and $8$. See Theorem~\ref{thmh1order} below, due to Seifert.

\begin{cor}\label{thmexceptional}
Any finite, non-cyclic surgery on a hyperbolic knot has surgery coefficient at least 7.
\end{cor}

Any Seifert fibered spaces which are not knot surgeries must be found among the dihedral manifolds, those with $|H_1(Y)|$ a multiple of 4 (see Corollary~\ref{thmh1order}). We will prove:

\begin{thm}\label{thmz32}
The following manifolds have unique surgery descriptions:
\[\begin{array}{rcl}
S^3_4\left(T_{3,2}\right) & = & \left(-1; \frac{1}{2}, \frac{1}{2}, \frac{1}{3}\right) \\
-S^3_8\left(T_{3,2}\right) & = & \left(-1; \frac{1}{2}, \frac{1}{2}, \frac{2}{3}\right)\\
S^3_{16/3}\left(T_{3,2}\right) & = & \left(-1; \frac{1}{2}, \frac{1}{2}, \frac{4}{3}\right) \\
-S^3_{20/3}\left(T_{3,2}\right) & = & \left(-1; \frac{1}{2}, \frac{1}{2}, \frac{5}{3}\right)\\
S^3_{28/5}\left(T_{3,2}\right) & = & \left(-1; \frac{1}{2}, \frac{1}{2}, \frac{7}{3}\right) \\
-S^3_{32/5}\left(T_{3,2}\right) & = & \left(-1; \frac{1}{2}, \frac{1}{2}, \frac{8}{3}\right) \\
-S^3_{32/3}\left(T_{5,2}\right) & = & \left(-1; \frac{1}{2}, \frac{1}{2}, \frac{8}{5}\right) \\
\end{array}\]

The only other dihedral manifolds with $p\leq32$ which may be surgery are:
\[\begin{array}{rcl c rcl}
S^3_8\left(T_{5,2}\right) & = & \left(-1; \frac{1}{2}, \frac{1}{2}, \frac{2}{5}\right) \\
-S^3_{12}\left(T_{5,2}\right) & = & \left(-1; \frac{1}{2}, \frac{1}{2}, \frac{3}{5}\right) \\
S^3_{12}\left(T_{7,2}\right) & = & \left(-1; \frac{1}{2}, \frac{1}{2}, \frac{3}{7}\right) \\
-S^3_{16}\left(T_{7,2}\right) & = & \left(-1; \frac{1}{2}, \frac{1}{2}, \frac{4}{7}\right) \\
S^3_{16}\left(T_{9,2}\right) & = & \left(-1; \frac{1}{2}, \frac{1}{2}, \frac{4}{9}\right)  \\
-S^3_{20}\left(T_{9,2}\right) & = & \left(-1; \frac{1}{2}, \frac{1}{2}, \frac{5}{9}\right)  \\
S^3_{20}\left(T_{11,2}\right) & = & \left(-1; \frac{1}{2}, \frac{1}{2}, \frac{5}{11}\right)  \\
-S^3_{24}\left(T_{11,2}\right) & = & \left(-1; \frac{1}{2}, \frac{1}{2}, \frac{6}{11}\right) \\
S^3_{24}\left(T_{13,2}\right) & = &  \left(-1; \frac{1}{2}, \frac{1}{2}, \frac{6}{13}\right) \\
S^3_{28/3}\left(T_{5,2}\right) = S^3_{28}\left(K_0\right)?& = & \left(-1; \frac{1}{2}, \frac{1}{2}, \frac{7}{5}\right)\\
S^3_{28}\left(K_1\right) & = & \left(-1; \frac{1}{2}, \frac{1}{2}, \frac{7}{11}\right)\\
-S^3_{28}\left(T_{13,2}\right) & = & \left(-1; \frac{1}{2}, \frac{1}{2}, \frac{7}{13}\right) \\
S^3_{28}\left(T_{15,2}\right) & = & \left(-1; \frac{1}{2}, \frac{1}{2}, \frac{7}{15}\right) \\
-S^3_{32}\left(T_{15,2}\right) & = & \left(-1; \frac{1}{2}, \frac{1}{2}, \frac{8}{15}\right) \\
S^3_{32}\left(T_{17,2}\right) & = & \left(-1; \frac{1}{2}, \frac{1}{2}, \frac{8}{17}\right) \\
\end{array}\]
The latter manifolds may also be integral surgery on hyperbolic knots with the same $\Delta_K(T)$ as the knots listed above; see Tables~\ref{tablealexpoly} and \ref{tablesurgcorrterms}. 
\end{thm}

Note: $K_1$ is the knot constructed by $+1$ surgery on the unknotted component of the $(-2, 3, 10)$ pretzel link \cite[Proposition~18]{bleilerhodgson2}. $K_0$ may be some knot with symmetrized Alexander polynomial $T^8-T^7+T^5-T^4+T^2-T+1 \cdots$; the author is not currently aware of any such $K_0$ with the listed surgery.

\begin{cor}
If $m\leq 8$, the following manifolds cannot be realized as knot surgeries: 
\[\begin{array}{rl}
\left(-1; \frac{1}{2}, \frac{1}{2}, \frac{m}{n}\right) & if~n > 2m+1.
\end{array}\]
\end{cor}

We describe Seifert fibered spaces and their non-hyperbolic surgeries in Section~\ref{sectexamples}, we list the necessary prerequisites about  L-space surgeries and the $d$-invariants in Section~\ref{sectd}, and we prove Theorems~\ref{thmz2} and \ref{thmz32} in Section~\ref{sectresults}.

The first presentation of this work may be found in the author's thesis~\cite{doigthesis}. Special thanks go to Zolt\'an Szab\'o, who suggested and directed this project with a great deal of kindness and patience, and to Chuck Livingston, Paul Kirk, and Dave Gabai, who listened and contributed many helpful suggestions. Thanks also to Stephen Maderak for turning my algorithms into functional code and making it possible to generate gigabytes of examples.

\section{Seifert fibered spaces as knot surgeries}\label{sectexamples}

Any closed, oriented 3-manifold $Y$ is surgery on some link in $S^3$ \cite{lickorish,wallace}. A surgery diagram can be manipulated by the methods of Kirby calculus \cite{kirbycalculus}, which alter the diagram but not the diffeomorphism type of the underlying 3-manifold: isotopy by surgery diagrams; stabilizing or destabilizing the manifold by adding or subtracting  a $\pm1$-framed unknot which can be separated from the rest of the link; and handlesliding one link component over another, replacing $L_2$ with the band sum of $L_1$ and $L_2$. For the last, if $n_i$ is the framing on $L_i$, then $n_2$ becomes $n_1+n_2+2 lk(L_1, L_2).$ For a presentation of Kirby calculus, including its applications to Dehn surgery, see, e.g., \cite[Chapter 5]{gompfstipsicz}. 

\paragraph{Seifert fibered spaces.}

Seifert fibered spaces were originally defined by Seifert in 1932 (\cite{seifert}, translated by W. Heil in \cite{seifertthrelfall}). Scott gives a more modern presentation with a slightly expanded definition incorporating the fibered solid Klein bottles mentioned below \cite{scott}.

A trivial solid torus $\{z\in \mathbb{C} : |z| \leq 1\}\times S^1$ may be given the product fibration with fibers $\{z\}\times S^1$. A \emph{fibered solid torus} (or \emph{fibered solid Klein bottle}) is a torus (or Klein bottle) which is finitely covered by the trivial fibered torus where the covering map preserves fibers. A fibered torus can alternately be constructed by taking the trivial fibered torus, cutting it along $\{z\in \mathbb{C} : |z| \leq 1\} \times \{0\}$, and identifying $(z,0)$ with $(e^{2\pi iq/p}z, 1)$, and a fibered solid Klein bottle can be constructed by taking the same cut fibered torus and identifying $(z,0)$ with $(\bar{z},1)$. The torus then has one \emph{exceptional} (not regular) fiber in the center, and the Klein bottle has a continuous family of exceptional fibers whose union is an annulus.

 A \emph{Seifert fibered space} is a manifold foliated by circles so that any circle has a neighborhood which is fiber isomorphic to a fibered solid torus or Klein bottle. A Seifert fibered space itself can be thought of as a fiber bundle over the orbifold obtained by compressing each fiber to a point (often called the \emph{base orbifold}). Each isolated exceptional fiber corresponds to a cone point on the orbifold and a surface of exceptional fibers corresponds to a reflector line in the orbifold. Each isolated exceptional fiber can be eliminated by some Dehn surgery, and the class of such surgery coefficients is referred to as the fiber's \emph{framing}. 

For our purposes, we will need only Seifert fibered spaces with base orbifold $S^2$ and some number of cone points. Construct such a space by choosing a circle bundle $\zeta$ over $S^2$ and surgering over fibers with framings $-b_i/a_i$ (the negative sign is for historical reasons). It can be described as surgery on a link in $S^3$ whose components have framing  \mbox{$\{b=c_1(\zeta), -b_1/a_1,\dots,-b_r/a_r\}$}. Seifert identified such a manifold with an $n$-tuple (together with information about the base orbifold which we will exclude):
\[\left(b; \frac{a_1}{b_1}, \dots, \frac{a_r}{b_r}\right)\]
For example, the Poincar\'e homology sphere is $-(-1; \frac{1}{2},\frac{1}{3},\frac{1}{5})$. 

The choice of framings is not unique. The $b_i$, sometimes called the \emph{multiplicities}, are determined, but $b$ and the $a_i$ may be altered by handleslides. For example,
\begin{eqnarray}\left(-b; \frac{a_1}{b_1}, \frac{a_2}{b_2},\frac{a_3}{b_3}\right) \cong \left(-b-1; \frac{a_1}{b_1}+1,\frac{a_2}{b_2},\frac{a_3}{b_3}\right)\label{eqnneum1}\end{eqnarray}
\begin{eqnarray}\left(-1;\frac{1}{2},\frac{a_2}{b_2},\frac{a_3}{b_3}\right) \cong -\left(-2;1-\frac{1}{2},1-\frac{a_2}{b_2},1-\frac{a_3}{b_3}\right).\label{eqnneum2}\end{eqnarray}

By geometrization \cite{perelman}, the manifolds with finite fundamental group are all Seifert fibered. They fall into five classes depending on whether $\pi_1$ is cyclic or is based on one of the four isometries of a sphere. We slightly rephrase Seifert's result:

\begin{thm}[Seifert \cite{seifert}] The closed, oriented Seifert fibered spaces with finite but non-cyclic fundamental group are exactly those manifolds with base orbifold $S^2$ and the following presentations:
\begin{enumerate}
\item Type I, icosahedral: $\Big(b; \frac{a_1}{2},\frac{a_2}{3},\frac{a_3}{5}\Big)$ with $H_1(Y)=\Zed_m$ and $(m,30)=1$.
\item Type O, octahedral:  $\Big(b; \frac{a_1}{2},\frac{a_2}{3},\frac{a_3}{4}\Big)$ with $H_1(Y)=\Zed_{2m}$ and $(m,6)=1$.
\item Type T, tetrahedral: $\Big(b; \frac{a_1}{2},\frac{a_2}{3},\frac{a_3}{3}\Big)$ with $H_1(Y)=\Zed_{3m}$ and $(m,2)=1$.
\item Type D, dihedral: $\Big(b;\frac{a_1}{2},\frac{a_2}{2},\frac{a_3}{b_3}\Big)$ with $H_1(Y)=\Zed_{4m}$ and $(m,b_3)=1$ (if $b_3$ is even) or $H_1(Y) = \Zed_2\times\Zed_{2m}$ with $(m,2b_3)=1$ (if $b_3$ is odd).
\end{enumerate}
where $|H_1(Y)| = b_1b_2b_3\left(b+\frac{a_1}{b_1} + \frac{a_2}{b_2} + \frac{a_3}{b_3}\right)$ and $(a_i, b_i) = 1$. Any integer $m$ meeting the constraints listed for one of the four types I, O, T, or D corresponds (up to orientation) to a unique Seifert fibered space of type I, O, or T, or to a unique infinite family of type D indexed by the integer $b_3$.

 Any choice of $b, a_i,$ and $b_i$ meeting the appropriate relative primality conditions gives a Seifert fibered space. For each orientation, we choose a canonical presentation where $b=-1$, $a_1=1$, and $a_2=1$ or $2$. If we allow change of orientation, we can also require $a_2=1$ and (for type D) $a_3>0$. \label{thmh1order}\end{thm} 
\begin{proof}
Seifert calculates explicit descriptions of the fundamental group and first homology group and then deduces the possible framings; see \cite{seifert} for the details.
\[\pi_1(Y) = 
\left<\lambda, \mu, \mu_1, \dots, \mu_r \left| 
\begin{array}{c}
\mu \mu_1 \cdots \mu_r = 1, \left[ \lambda, \mu_i \right] = 1, 
\\
\mu=\lambda^b, ~ \mu_i^{\phantom{i}b_i} = \lambda^{a_i},
\end{array}\right.\right>
\]
\[\begin{array}{c}H_1(Y;\Zed) =\\ \phantom{ }\end{array}
 \begin{array}{c}
\Zed m_0 \oplus \cdots \oplus \Zed m_r \\
\hline
\left(\begin{array}{c}
	b \cdot m_0 +  m_1 + \cdots + m_r = 0 \\
	a_i \cdot m_0 = b_i \cdot m_i 
	\end{array}
\right)
\end{array}\]

To obtain the canonical presentation for a manifold of type D, first turn $(b;\frac{a_1}{2},\frac{a_2}{2},\frac{a_3}{b_3})$ into $(b;\frac{1}{2},\frac{1}{2},\frac{a_3'}{b_3})$ using Equation~\ref{eqnneum1}. Then adjust $b$ (perhaps changing $a_3'$ but leaving $a_1=a_2=1$) to get $(-1;\frac{1}{2},\frac{1}{2},\frac{a_3''}{b_3})$. If $a_3''<0$, reverse orientation as in Equation~\ref{eqnneum2} to $-(-2;\frac{1}{2},\frac{1}{2},\frac{b_3-a_3''}{b_3})=-(-1;\frac{1}{2},\frac{1}{2},\frac{-a_3''}{b_3})$ For types I, O, and T, first obtain $(-1; \frac{1}{2},\frac{1~or~2}{3},\frac{a_3}{b_3})$. If $a_2=2$, reverse orientation to $-(-2; \frac{1}{2},\frac{1}{3},1-\frac{a_3}{b_3})=-(-1; \frac{1}{2},\frac{1}{3},-\frac{a_3}{b_3})$.

Given a choice of I, O, T, or D and an $m$ that meets the appropriate primality conditions, the $b_i$ are determined, and there are $b$ and $a_i$ as follows. Assume $b=-1$ and $a_1=a_2=1$. For type I with $(m,30)=1$, $m\pmod 6\equiv -5$ or $5$. In the former case, set $a_3=\frac{m+5}{6}$, and $|H_1(Y)|=|6a_3-5|=m$; in the latter, set $a_3=-\frac{m-5}{6}$, so $|H_1(Y)|=m$. For type O with $(m,24)=1$, then $(m,3)=1$, so choose $a_3 = \frac{m+2}{3}$ or $-\frac{m-2}{3}$, whichever is an integer, and then $|H_1(Y)|=|6a_3-4|=2m$. For type T with $(m,18)=1$, then $(m,3)=1$, so choose $a_3 = \frac{m+1}{2}$ or $-\frac{m-1}{2}$, whichever is an integer, and then $|H_1(Y)|=|6a_3-3|=3m$. Finally, for type D with $(m,b_3)=1$, choose $a_3 = m$, so $|H_1(Y)|=4m$. 

Note that the canonical presentations for the two orientation of an I, O, or T manifold may be distinguished by whether $a_2$ is $1$ or $2$. The two orientations for a D manifold may be distinguished by whether $a_3=m$ is positive or negative.
\end{proof}

\paragraph{Finite surgeries.}

Many of the elliptic manifolds can be realized as torus knot surgeries.
\begin{thm} \cite[Propositions 3.1, 3.2, 4]{moser} \label{thmmoser}
\[
S^3_{p/q}(T_{r,s})=
\left\{\begin{array}{ll}
L_{r,s}\# L_{s,r} & if ~p/q=rs\\
L_{p,rsq} & if~ p/q=rs\pm 1/q\\
\left(b; \frac{a_1}{r}, \frac{a_2}{s}, \frac{a_3}{|rsq-p|}\right) & otherwise,~for~some~choice~of~ b,a_1,a_2,a_3
\end{array}\right.
\]
\end{thm}

\begin{cor}\label{thmIOTaresurgeries}
Every manifold of type I, O, or T is surgery on a $T_{n,2}$ torus knot. Of each infinite family of manifolds of type D with the same $|H_1(Y)|=4m$, only finitely many are surgeries on torus knots, and they are the ones where $b_3$ divides $2m+1$ or $2m-1$.
\end{cor}

\begin{proof}
A careful application of Kirby calculus shows:
\[\begin{split}
S^3_{p/q}(T_{s,2}) 
&= \left(-1; \frac{1}{2}, \frac{(s-1)/2}{s}, \frac{q}{2sq-p}\right) 
= -\left(-1; \frac{1}{2}, \frac{(s+1)/2}{s}, \frac{q}{p-2sq}\right)\\
S^3_{p/q}(T_{4,3}) 
&= \left(-1; \frac{2}{3}, \frac{1}{4}, \frac{q}{12q-p}\right) 
= -\left(-1; \frac{1}{3}, \frac{3}{4}, \frac{q}{p-12q}\right)\\
S^3_{p/q}(T_{5,3}) 
&= \left(-1; \frac{1}{3}, \frac{3}{5}, \frac{q}{15q-p}\right) 
= -\left(-1; \frac{2}{3}, \frac{2}{5}, \frac{q}{p-15q}\right)
\end{split}\]

These cases cover all the finite torus knot surgeries since $p/q$-surgery on $T_{r,s}$ (if it is not a lens space or sum of lens spaces) has multiplicities $(r, s, |rsq-p|)$. A type I manifold may be surgery on $T_{3,2}$, $T_{5,2}$, or $T_{5,3}$; a type O manifold may be surgery on $T_{3,2}$ or $T_{4,3}$; a type T may be surgery on $T_{3,2}$; and a type D may be surgery on $T_{n,2}$.

By Theorem~\ref{thmh1order}, any I, O, or T manifold $Y$ may be written $\pm(-1; \frac{1}{2},\frac{1}{3},\frac{a_3}{b_3})$. A series of blow-ups on the trefoil shows $Y$ is $\frac{6a_3-b_3}{a_3}$-surgery on $T_{3,2}$ (up to orientation). 

A manifold of type D with multiplicities $(b_1,b_2,b_3)=(2,2,n)$ can only be surgery on a knot if $H_1(Y)$ is cyclic, meaning $n$ is odd, and it can only be surgery on $T_{n,2}$ if $2sq-p=\pm 2$. (\emph{NB:} $\frac{q}{2sq-p}$ is a reduced fraction since $(p,q)=1$.) Then $p=|H_1(Y)|$ and $q=\frac{|H_1(Y)|\pm2}{2n}$, i.e., $n$ divides either $\frac{|H_1(Y)|}{2}+1$ or $\frac{|H_1(Y)|}{2}-1$.
\end{proof}

\section{The invariant $d(Y,\spincstr)$}\label{sectd}

Heegaard Floer homology assigns a set of invariants (in our case, a graded abelian group over $\Zed_2$) to a closed, connected, oriented 3-manifold using a Heegaard decomposition of the manifold \cite{ozsz1,ozsz2}. A Langrangian Floer homology starts with a $2n$-dimensional symplectic manifold and  two $n$-dimensional Lagrangian submanifolds which meet transversely. The chain complex is a free $R$-module (for $R=\Zed_2$, $\Zed$, etc.) whose generators come from intersection points of the Lagrangians and whose boundary map counts pseudo-holomorphic disks associated to pairs of generators. Heegaard Floer homology is a Floer homology (after the work of Perutz \cite{perutzhandleslides}) that defines the symplectic manifold and the Lagrangians using a Heegaard decomposition of a 3-manifold. The generators of its chain complex can be thought of as sets of points on the Heegaard surface and the boundary maps can be analyzed by examining domains in the surface.

Heegaard Floer homology assigns a set of invariants to certain 3-manifolds $Y$, including rational homology spheres, indexed by their $\spinc$ structures $\spincstr$. These invariants are called the \emph{correction terms} or \emph{d-invariants} $d(Y,\spincstr)$. The hat version $\hf(Y)$ comes with a relative $\Zed$-grading which lifts to an absolute $\mathbb{Q}$-grading for a rational homology sphere (see \mbox{Theorem 7.1} of \cite{ozszsmooth}); it is defined by requiring that $\hf(S^3)\cong \Zed$ is supported in degree 0 and that the inclusion map $\widehat{CF}(Y,\spincstr) \hookrightarrow CF^+(Y,\spincstr)$ preserves degree. Then $d(Y,\spincstr)$ is the minimal grading of any non-torsion
class in $HF^+(Y,\spincstr)$ coming from $HF^{\infty}(Y,\spincstr)$ \cite{ozszabsolute}. If $Y$ is elliptic, all classes in $HF^+(Y)$ come from $HF^{\infty}(Y)$, and $d(Y,\spincstr)$ is defined for all $\spincstr$.

\paragraph{L-space surgeries.}
An elliptic Seifert fibered space is an example of an L-space, the Heegaard Floer homology version of a lens space \cite{ozszplumbed}. $\hf(Y)$ splits into $\oplus^{\spincstr} \hf(Y,\spincstr)$ over $\spinc$ structures (equivalence classes of non-zero vector fields $\spincstr$ that form a torsor over $H^2(Y;\Zed)$). Lens spaces have the nice property that each generator of $\hf(L(p,q))$ falls into a different torsion $\spinc$ structure ($\spincstr\in \spinc(Y)$ is torsion if $PD(c_1(\spincstr)) \in H_1(Y)$ is torsion). We will call any rational homology sphere with this property an \emph{L-space}. Equivalently, $\hf(Y,\spincstr) \cong\hf(S^3)$ for all $\spincstr$.

Using the surgery exact sequences and absolute grading on $HF^+(Y)$, we can place some restrictions on which knots may have L-space surgeries. Normalize the Alexander polynomial so
\[\Delta_K(T) = a_0 + \sum_{i=1}^n a_i(T^i + T^{-i}).\]
\begin{thm}\cite[Corollary 1.3]{ozszlens}
If a knot $K\subset S^3$ admits an L-space surgery, then the non-zero coefficients of $\Delta_K(T)$ are alternating $+1$s and $-1$s.\label{thmalex}
\end{thm}

Ozsv\'ath and Szab\'o \cite{ozszgenus} showed that the knot Floer homology $\widehat{HFK}(K,i)$ is $\Zed$ in the top grading $i=g(K)$ for any fibered knot, and Ghiggini \cite{ghigginigenusonefibered} and Ni \cite{nifibered} and, independently, Juh\'asz \cite{juhaszsutured} showed the converse; since $\Delta_K(T)$ is the graded Euler characteristic of $\widehat{HFK}(S^3,K)$, this means that
\begin{cor}\cite[Corollary 1.3]{nifibered}
If a knot $K\subset S^3$ admits an L-space surgery, then $K$ is fibered.
\label{thmfibered}\end{cor}

Finally, 
\begin{thm}\cite[Corollary 1.4]{ozszqsurgery} If a non-trivial knot $K$ admits a positive L-space surgery, then $S^3_{p/q}(K)$ is an L-space if and only if
\[\frac{p}{q} \geq 2 g(K) -1.\]\label{thmsurgcoeffform}
\end{thm}

These facts lead to another observation which seems to be known among the community but not frequently written down. 

\begin{cor}\label{thmamphichiral}
No non-trivial knot has both positive and negative L-space surgeries. No amphichiral knots have L-space surgeries. In particular, no knot has both positive and negative finite surgeries, and no amphichiral knot has any finite surgeries.
\end{cor}

\begin{proof}
If $K$ has a positive L-space surgery, then $\tau(K)=\deg(\Delta_K(T))=g(K)$ \cite[Corollary 1.6]{ozszlens}. If $K$ has both positive and negative L-space surgeries, meaning both $K$ and its mirror $mK$ have positive L-space surgeries, then $\tau(K)=g(K)=g(mK)=\tau(mK)$, but $\tau(K)=-\tau(mK)$.
\end{proof}

\paragraph{Calculating $d(Y,\spincstr)$ of a knot surgery.}
If $\surg(K)$ is an L-space, then $\hf(\surg(K))$ and its gradings can be calculated from $\Delta_K(T)$ and $p/q$:
\begin{thm} If $0<q<p$, there is a particular identification of $\spinc$ structures with $\Zed_p$ such that
\begin{itemize}
\item[a.] \cite[Proposition 4.8]{ozszabsolute}:
for $0 \leq i < p+q$, 
\[d(\surg(U),i) = -\left( \frac{p q-(2i+1-p-q)^2}{4p q}\right) - d(S^3_{q/r}(U),j)\]
where $r \equiv p \bmod q$ and $j \equiv i \bmod q$.
\item[b.] \cite[Theorem 1.2]{ozszqsurgery}: for $|i|\leq \frac{p}{2}$,
\[d(S^3_{p/q}(K),i) - d(S^3_{p/q}(U),i) = -2 \sum_{j=1}^{\infty} ja_{c+j}\]
where $c=\left|\left\lfloor \frac {i}{q} \right\rfloor\right|$ and the $a_j$ are the coefficients of the symmetrized Alexander polynomial.
\end{itemize}\label{thmd}
\end{thm}

\paragraph{Calculating $d(Y,\spincstr)$ of a Seifert fibered space.}\label{sectionplumbing}
It is often possible to calculate the $d(Y,\spincstr)$ algorithmically using plumbing graphs \cite{ozszplumbed}.

See \cite{neumannplumbing} for a thorough exposition of plumbing graphs. 

Consider $\Gamma$ a tree with vertices $\mathfrak{v}$ which have integer weights $\mathrm{m}(\mathfrak{v})$. The graph $\Gamma$ describes a 4-manifold $X=\mathrm{X}(\Gamma)$: for each vertex, take a disk bundle over the sphere with Euler number $\weight(\mathfrak{v})$; for each edge, plumb together the corresponding bundles. The boundary of $\mathrm{X}(\Gamma)$ is a 3-manifold we call $\mathrm{Y}(\Gamma)$. For example, the lens space $L(7,4)$ may be given as $Y(\Gamma)$ for either graph below since $[-3, -2, -2]$ and $[-2,3]$ are both continued fraction expansions for $-7/3$:
\[\xymatrix{
\stackrel{-3}{\bullet} \ar@{-}[r] & \stackrel{-2}{\bullet} \ar@{-}[r] & \stackrel{-2}{\bullet} \qquad = \qquad  \stackrel{-2}{\bullet} \ar@{-}[r] & \stackrel{+3}{\bullet}}\] 
Similarly, the Poincar\'e homology sphere (with nonstandard orientation) $Y=\left(-2; \frac{1}{2}, \frac{2}{3}, \frac{4}{5}\right)$ is given by
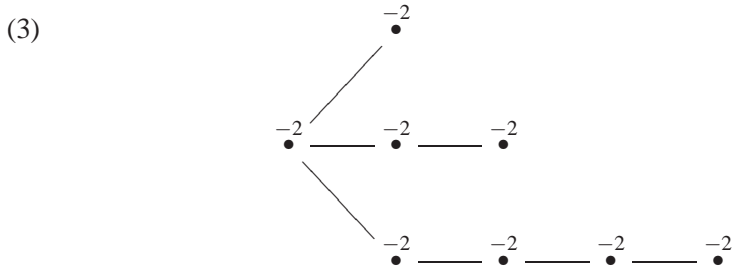
\begin{equation}
\xymatrix{
& \stackrel{-2} {\bullet} \ar@{-}[dl] \\
\stackrel{-2}{\bullet}\ar@{-}[dr]  \ar@{-}[r] & \stackrel{-2}{\bullet} \ar@{-}[r] & \stackrel{-2}{\bullet} \\
 & \stackrel{-2}{\bullet} \ar@{-}[r] & \stackrel{-2}{\bullet} \ar@{-}[r]& \stackrel{-2}{\bullet}\ar@{-}[r] & \stackrel{-2}{\bullet}
}\label{eqngraph}
\end{equation} 
where the central vertex is $\mathfrak{v}_1$; the top arm is a single vertex $\mathfrak{v}_2$ of weight $-2/1$; the next arm consists of vertices $\mathfrak{v}_3, \mathfrak{v}_4$ labelled from left to right with weights giving the continued fraction expansion of $-3/2$; and the bottom arm is $\mathfrak{v}_5, \cdots, \mathfrak{v}_8$ with weights giving the continued fraction expansion of $-5/4$. In general, an elliptic space may be written $Y=(b; \frac{a_1}{b_1},\frac{a_2}{b_2}, \frac{a_3}{b_3})$ with $0 < a_i/b_i < 1$ and (perhaps after reversing orientation) $b \leq 0$. If $p/q<0$, it is possible to chose the $x_i\leq -2$ (choose them to be negative; if a $-1$ appears, blow it down). Then $Y=Y(\Gamma)$ for a graph $\Gamma$ with a central vertex of degree $3$ and weight $b$ and with three arms with vertices of degree $\leq 2$ with weights given by the continued fraction expansion of $-b_i/a_i$ chosen so that the weights are $\leq -2$. Additionally, if the orientation is chosen so that $e(\Gamma)=b-\sum_{i=1}^3 a_i /b_i <0$ (i.e., $b\leq 0$), then $\Gamma$ is the dual graph of a good resolution of a singularity and $X(\Gamma)$ is negative definite \cite[Corollary 8.3]{neumannplumbing}. 

An elliptic space $Y$ with the description given above additionally has the property that $\weight(\mathfrak{v}) \leq -\deg(\mathfrak{v})$ for each vertex except possibly the central one, as in (\ref{eqngraph}); we will call a vertex violating this property \emph{bad}.

$H_2(X;\Zed)$ is a lattice freely spanned by the vertices of $\Gamma$. Define a matrix $Q$ for the intersection form using $\Gamma$: if $\mathfrak{v}$ is a vertex and $v$ the corresponding homology class, $v\cdot v=\weight(\mathfrak{v})$; if $\mathfrak{v}$ and $\mathfrak{w}$ are distinct vertices, $v\cdot w=1$ if there is an edge between $\mathfrak{v}$ and $\mathfrak{w}$ and $0$ otherwise. If, as above, $e(\Gamma)<0$, then $Q$ is negative definite. For $\Gamma$ in (\ref{eqngraph}), $Q$ is $E_8$.

The \emph{characteristic vectors} or $\Char(\Gamma)$ are the $V\in H^2(X;\Zed)$ such that 
\[\langle V, w \rangle \equiv w\cdot w \bmod 2 \quad \forall ~w\in H_2(X;\Zed).\]
$\Char(\Gamma)$ splits over $\spinc(Y(\Gamma))$. Let $\Char_\spincstr(\Gamma)$ be the characteristic vectors where $V=c_1(\spincstrextend)$ for some $\spincstrextend\in \spinc(X(\Gamma))$ with $\spincstrextend|_{Y(\Gamma)}=\spincstr$. It is easy to identify a characteristic vector using Hom duality: for $V\in H^2(X)$, note that $\langle V,w\rangle = PD^{-1}(V)\cdot w = v^TQw$ for some $v\in H_2(X)$. Then $v^T$ is the Poincar\'e dual of $V$, and $v^TQ$ is its Hom dual. $V$ is characteristic exactly when $PD^{-1}(V)\cdot v_i \equiv v_i \cdot v_i \bmod 2$, i.e., the $i^{th}$ coordinate of $v^TQ$ has the same parity as $\weight(\mathfrak{v}_i)$ for all vertices $\mathfrak{v}_i$. For example, $\Char(\Gamma)$ of (\ref{eqngraph}) consists of all vectors $v^TQ$ with even coordinates.

$HF^+(-Y,\spincstr)$ can be expressed in terms of $\Char_\spincstr(\Gamma)$. Let
\[\Tnaught = \Zed[U,U^{-1}]/U\cdot \Zed[U]\]
as a $\Zed[U]$-module with grading so that $U^{-d}$ is homogeneous and supported in degree $2d$ (where $d>0$). Then $HF^+(-Y,\spincstr)$ is isomorphic to the set of functions
\[\phi: Char_\spincstr(G) \rightarrow \Tnaught\]
which preserve the adjunction relations
\begin{align*}
U^n \cdot \phi(V + PD(w))& = \phi(V) &\textrm{ if } n\geq 0\\
\phi(V + PD(w)) &= U^{-n} \cdot\phi(V) &\textrm{ if } n\leq 0
\end{align*}
where $2n=\langle V, w \rangle + w \cdot w$. 

The grading of $HF^+(-Y,\spincstr)$ is induced from the grading on $\Tnaught$ by
\[\deg(\phi) = \deg(\phi(V)) - \frac{V^2 + |\Gamma|}{4}\]
if $\phi(V)\in \Tnaught$ is a non-trivial homogeneous element, where $|\Gamma|$ is the number of vertices in $\Gamma$. We could calculate $d(Y,\spincstr)$ by optimizing this grading over $\Char_\spincstr(\Gamma)$, but it would be very labor intensive. To better study the grading on characteristic vectors, define an operation on $\Char_\spincstr(\Gamma)$ by 
\begin{equation}\label{eqnoperation} V \mapsto V+2PD(v_i) \qquad if~ \langle V,v_i\rangle = -\weight(\mathfrak{v}_i) \end{equation}
That is, find $v^TQ$ where $v^TQ$ has $-\weight(\mathfrak{v}_i)$ as its $i^{th}$ coordinate and has the same parity as $\weight(\mathfrak{v}_j)$ in all other coordinates. This operation changes the $i^{th}$ coordinate to $\weight(\mathfrak{v}_i)$ and adds 2 to the $j^{th}$ coordinate if and only if there is a edge between $\mathfrak{v}_i$ and $\mathfrak{v}_j$. This operation does not change the class in $\Char_\spincstr(\Gamma)$, and it does not change the value $V^2=\langle V, PD^{-1}(V)\rangle$. For the graph $\Gamma$ of (\ref{eqngraph}), the vector 
\[V=(2, 0, 0, 0,0,0,0,0)\] satisfies $ \langle V,v_i\rangle = -\weight(\mathfrak{v}_i)$  for $i=1$, so the operation gives
\[V+2PD(v_1)= (-2, 2, 2, 0,2,0,0,0)\] 
and \[V'=(-2, 2, 2, 0,2,0,0,0)\] satisfies the equality for $i=2, 3,$ or $5$, which gives
\[\begin{split}
V'+PD(v_2)=(0, -2, 2, 0,2,0,0,0)\\ V'+PD(v_3)=(0, 2, -2, 2,2,0,0,0)\\ V'+PD(v_5)=(0, 2, 2, 0,-2,2,0,0)\end{split}\]

A \emph{path} of vectors is a sequence $\{V_0, V_1, \cdots, V_k\}$ where $V_{i+1}$ is derived from $V_i$ by this operation, and a \emph{full} path is maximal with respect to this operation. For example, $\{(0,0,0,0,0,0,0,0)\}$ is actually a full path for $\Gamma$ in (\ref{eqngraph}).

A \emph{nice} characteristic vector obeys
\begin{equation}\label{eqnnice}\weight(\mathfrak{v_i}) \leq \langle V,v_i\rangle \leq -\weight(\mathfrak{v_i}) \quad \forall~i\end{equation}
that is, the $i^{th}$ coordinate of $v^TQ$ is between $-\weight(\mathfrak{v_i})$ and $\weight(\mathfrak{v_i})$. There are a finite number of nice characteristic vectors. By \cite[Proposition~3.2]{ozszplumbed}, every full path of \emph{nice} vectors $\{V_0, V_1, \cdots, V_k\}$ obeys the additional property that $V_0$ and $V_k$ obey 
\begin{eqnarray}
\weight(\mathfrak{v_i}) < \langle V_0,v_i\rangle \leq -\weight(\mathfrak{v_i}) \quad \forall~i\label{eqnnicefull}\\
\weight(\mathfrak{v_i}) \leq \langle V_k,v_i\rangle < -\weight(\mathfrak{v_i}) \quad \forall~i\label{eqnnicefull2}
\end{eqnarray}
For example, for $\Gamma$ in (\ref{eqngraph}), $(2,2,0,0,0,0,0,0)$ is nice, but $(4,-2,0,0,0,0,0,0)$ is not. For a given vector $V$, if there is any full path of nice vectors containing $V$, then all paths containing $V$ have only nice vectors, and all full paths containing $V$ are the same length and start and end at the same $V_0$ and $V_k$. For  $\Gamma$ of (\ref{eqngraph}), there is only one full path of nice vectors, and it contains only the vector $(0,0,0,0,0,0,0,0)$.

Using full nice paths, we may now calculate $d(Y(\Gamma),\spincstrextend)$ in a reasonably efficient fashion.

\begin{thm}\cite[Corollaries 1.5 and 3.2]{ozszplumbed} Let $\Gamma$ be a connected tree with at most one bad vertex and $\spincstr\in\spinc(Y(\Gamma))$. Then
\[d(Y(\Gamma),\spincstr) = -\max_{V \in \Char_\spincstr (\Gamma)}  \frac{V^2 +|\Gamma|}{4}\]
In fact, this maximum is obtained over the vectors that are part of nice full paths and obey Equation~\ref{eqnnicefull} (equivalently, Equation~\ref{eqnnicefull2}).
\end{thm}

Recall that $V$ is actually an element in $HF^+(-Y(\Gamma),\spincstr)$, hence the negative sign and maximum instead of minimum.

Given a characteristic vector $V$ written as $v^TQ$, it is easy to calculate $d(Y(\Gamma),\spincstr)$ since $V^2=VQV^T=(v^TQ)Q^{-1}(v^TQ)^T$. In the case of $\Gamma$ of (\ref{eqngraph}), the vector $v^TQ=(0,0,0,0,0,0,0,0)$ has $V^2=0$ and so $d(-Y(\Gamma),\spincstr_0)=2$ where $\spincstr_0$ is the single $\spinc$ structure.

\begin{ex}Calculate the correction terms for the first family of dihedral manifolds from Theorem~\ref{thmh1order}, $Y = \left(-1;\frac{1}{2},\frac{1}{2},\frac{1}{n}\right)$. Reversing orientation so $e(-Y)<0$, $-Y = -\left(-2;\frac{1}{2},\frac{1}{2},\frac{n-1}{n}\right)$ and $\Gamma$ is
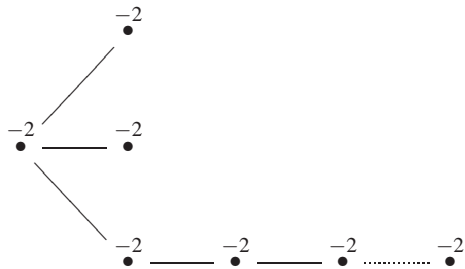
\[\xymatrix{
 &  \stackrel{-2} {\bullet} \\
 \stackrel{-2} {\bullet} \ar@{-}[ur]\ar@{-}[r] \ar@{-}[rd]&   \stackrel{-2} {\bullet}\\
 &  \stackrel{-2} {\bullet}\ar@{-}[r]  &   \stackrel{-2} {\bullet}\ar@{-}[r] &  \stackrel{-2} {\bullet} \ar@{..}[r] &  \stackrel{-2} {\bullet}
 }\]
where $\mathfrak{v}_1$ is degree $3$ vertex on the left, $\mathfrak{v}_2$ is on the top arm, $\mathfrak{v}_3$ is on the middle arm, and $\mathfrak{v}_4,\cdots, \mathfrak{v}_{n+2}$ are on the bottom arm.

Equivalently, the 4-manifold has intersection form 
\[Q=
\left[\begin{array}{cccccccc}
-2 & 1 & 1 & 1 & 0 & \cdots & 0 & 0\\
1 & -2 & 0 & 0 & 0 & \cdots & 0 & 0\\
1 & 0 & -2 & 0 & 0 & \cdots & 0 & 0\\
1 & 0 & 0 & -2 & 1 & \cdots & 0 & 0\\
0 & 0 & 0 & 1 & -2 & \cdots & 0 & 0\\
\vdots & \vdots & \vdots & \vdots & \vdots & \ddots & \vdots & \vdots\\
0 & 0 & 0 & 0 & 0 & \cdots & -2 & 1  \\
0 & 0 & 0 & 0 & 0 & \cdots & 1 & -2
\end{array}\right]\]
There are four full, nice paths for $\Gamma$; they start with the four vectors $V$ which, when written in the form $v^TQ$, are 
\[\begin{split}(0, 0, 0, 0, \cdots, 0)\\(0, 2, 0, 0, \cdots, 0)\\(0, 0, 2, 0, \cdots, 0)\\(0, 0, 0, 0, \cdots, 2)\end{split}\] These vectors have squares \[0,~ -(n+2),~ -(n+2), ~-4\] respectively, and so the correction terms of $Y(\Gamma)$ are \[-\frac{n+2}{4}, ~0, ~0,~ - \frac{n-2}{4}\]
\label{ex4}
\end{ex}

We list the correction terms for all the dihedral manifolds with $|H_1(Y)|\leq 32$ in Table~\ref{tableSFScorrterms}.

\section{$d(Y,\spincstr)$ as a knot surgery obstruction}\label{sectresults}

To demonstrate some of the techniques will use to prove Theorems~\ref{thmz2} and~\ref{thmz32}, we will summarize the proof of Theorem~\ref{thmz1}, due to Ghiggini:

\begin{proof}[Proof of Theorem~\ref{thmz1} {\cite[Corollary 1.7]{ghigginigenusonefibered}}]
By Theorem~\ref{thmh1order}, the Poincar\'e homology sphere is (up to orientation)
\[Y = -\left(-1; \frac{1}{2}, \frac{1}{3}, \frac{1}{5}\right).\]
Assume $Y$ or $-Y=S^3_{p/q}(K)$ with $p/q>0$. Recall $S^3_{p/q}(K) =-S^3_{-p/q}(mK)$ and $d(Y,\spincstr)=-d(-Y,-\spincstr)$. Then $|H_1(Y)|=p=1$. Since $Y$ is not a lens space, $g(K)>0$; since it is an L-space, Theorem~\ref{thmsurgcoeffform} says:
\[\frac{1}{q} \geq 2g(K)-1\]
Therefore, $p/q=1$ and $g(K)=1$. 

By Theorem~\ref{thmalex},
\[\Delta_K(T) = T-1+T^{-1}.\]
By Corollary~\ref{thmfibered}, $K$ is fibered. Therefore, $K$ is the right-handed trefoil $T_{3,2}$ or the left-handed one $T_{3,-2}$ (see, e.g., Burde and Zieschang \cite{burdezieschang}). By the calculations in the proof of Corollary~\ref{thmIOTaresurgeries}, $S^3_{1}(T_{3,2})=-S^3_{-1}(T_{3,-2})=\left(-1; \frac{1}{2}, \frac{1}{3}, \frac{1}{5}\right)$, and $S^3_{1}(T_{3,-2})=-S^3_{-1}(T_{3,2})=-\left(-1; \frac{1}{2}, \frac{1}{3}, \frac{1}{8}\right)$, which is not elliptic by Theorem~\ref{thmh1order}.
\end{proof}

\begin{table}[ht]
\[\begin{array}{c | cccccccccccc }
p & d  \\
\hline
1 & -2\\
2 & -\frac{7}{4} & -\frac{1}{4} \\
3 & -\frac{1}{6} & -\frac{3}{2} & -\frac{1}{6} \\
7 & \frac{1}{14} & -\frac{3}{14} & -\frac{19}{14} &-\frac{1}{2} &  -\frac{19}{14} & -\frac{3}{14} & \frac{1}{14} \\
9 & 0 & -\frac{10}{9} & -\frac{4}{9} & \frac{2}{9} & 0 &  \frac{2}{9} & -\frac{4}{9} & -\frac{10}{9} & 0\\
\end{array}\]
\caption{The correction terms for $Y$ for $p= |H_1(Y)| < 10$. (See Table~\ref{tableSFScorrterms} for $p=4$ and $8$.) $Y$ is given the canonical orientation as in Theorem~\ref{thmh1order}. }\label{tableSFScorrterms10}
\end{table}

\begin{proof}[Proof of Theorem~\ref{thmz2}]
Assume $Y$ is Seifert-fibered but not a lens space, and $Y=S^3_{p/q}(K)$ where $p/q>0$. In general, if $S^3_{p/q}(K)$ is an elliptic space, then $S^3_{-p/q}(mK)$ (where $mK$ is the mirror) is elliptic, too, but $S^3_{p/q}(mK)$ and $S^3_{-p/q}(K)$ are not; see Corollary~\ref{thmamphichiral}. 

{$\bf|H_1(Y)|=2$:}  $Y=\pm\left(-1; \frac{1}{2}, \frac{1}{3}, \frac{1}{4}\right)$ and
\[\frac{2}{q} \geq 2g(K)-1,\]
so $p/q=2$ with $g(K)=1$ and $\Delta_K(T)=T-1+T^{-1},$ and $K$ must be a trefoil. By Corollary~\ref{thmIOTaresurgeries}, $S^3_{2}(T_{3,2})=\left(-1; \frac{1}{2}, \frac{1}{3}, \frac{1}{4}\right)$.

$\bf|H_1(Y)|=3$:  $Y=\pm\left(-1; \frac{1}{2}, \frac{1}{3}, \frac{1}{3}\right)$. Either $p/q=3/2$ with $g(K)=1$ and $\Delta_K(T)=T-1+T^{-1}$ and $K$ the trefoil, or else $p/q=3$ with $0<g(K)\leq2$. In the latter case, Theorem~\ref{thmalex} shows that the symmetrized Alexander polynomial of $K$ may be $\Delta_1(T), \Delta_2(T),$ or $\Delta_{2'} (T)$ (see Table~\ref{tablealexpoly} for a number of the Alexander polynomials we will use).

To narrow this down, calculate the corresponding correction terms that would result from $+3$-surgery on knots with these Alexander polynomials. By Theorem~\ref{thmd}:
\begin{multline*}
d(S^3_3(K),i) = d(S^3_3(U),i)  -2\sum_{j=1}^{\infty} ja_{j+c}\\
= \left\{\begin{array}{cl} 1/2 & i = 0\\ -1/6 & i=\pm 1 \end{array}\right.
+\left\{\begin{array}{rrrl}
\phantom{-}\Delta_1 & \phantom{-}\Delta_{2} & \phantom{-}\Delta_{2'} \\
-2 & -2 & -4 & i=0\\
\phantom{-}0 & -2 & -2 & i=\pm1\\
\end{array}\right.\\
= \left\{\begin{array}{cccl}
\phantom{-}\Delta_1 &\phantom{-} \Delta_{2} & \phantom{-}\Delta_{2'} & \\
-3/2 & -3/2 & -7/2 & i = 0\\
-1/6 & -13/6 & -13/6 & i = \pm1\\
\end{array}\right.
\end{multline*}
On the other hand, $d(Y,\spincstr)$ may be calculated as in Example~\ref{ex4}; $-Y=\left(-2; \frac{1}{2}, \frac{2}{3}, \frac{2}{3}\right)$ has $e(-Y)<0$, and $\Gamma$ is
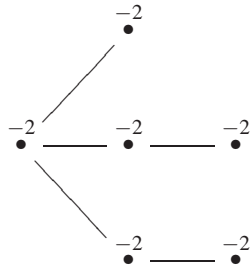
\[\xymatrix{
 &  \stackrel{-2} {\bullet}\\
 \stackrel{-2} {\bullet} \ar@{-}[ur]\ar@{-}[r] \ar@{-}[rd]&   \stackrel{-2} {\bullet}\ar@{-}[r] &   \stackrel{-2} {\bullet}\\
 &  \stackrel{-2} {\bullet}\ar@{-}[r]  &   \stackrel{-2} {\bullet}
 }\]
where $\mathfrak{v}_1$ is degree $3$ vertex on the left, $\mathfrak{v}_2$ is on the top arm, $\mathfrak{v}_3, \mathfrak{v}_4$ are on the middle arm, and $\mathfrak{v}_5, \mathfrak{v}_6$ are on the bottom arm.

There are three nice full paths, starting with the vectors (written in the form $v^TQ$) 
\[\begin{split}(0,0,0,0,0,0)\\ (0,0,0,2,0,0)\\ (0,0,0,0,0,2)\end{split}\] with squares \[0,~ \frac{16}{3},~ \frac{16}{3},\] so the correction terms $d(Y(\Gamma),\spincstr)$ are, in some order, \[-\frac{3}{2}, ~-\frac{1}{6}, ~-\frac{1}{6}\] 
These terms do not match the correction terms coming from surgery on a knot with Alexander polynomial $\Delta_{2}(T)$ or $\Delta_{2'}(T)$, so the Alexander polynomial can only be $\Delta_1(T)$, and $K$ is a trefoil. Note $S^3_{3}(T_{3,2})=\left(-1; \frac{1}{2}, \frac{1}{3}, \frac{1}{3}\right)$.

\begin{table}\[\begin{array}{l}
\Delta_{1\phantom{'}}(T)  =  T-1 \cdots \\
\Delta_{2\phantom{'}}(T) = T^2-T+1 \cdots \\
\Delta_{2'}(T) = T^2-1 \cdots \\
\Delta_{3\phantom{'}}(T) = T^3-T^2+T-1 \cdots \\
\Delta_{4\phantom{'}}(T) = T^4-T^3+T^2-T+1 \cdots \\
\Delta_{4'}(T) = T^4 -T^3+T-1 \cdots \\
\Delta_{5\phantom{'}}(T) = T^5-T^4+T^3-T^2+T-1 \cdots \\
\Delta_{6\phantom{'}}(T) = T^6-T^5+T^4 -T^3+T^2-T+1 \cdots \\
\Delta_{7\phantom{'}}(T) = T^7-T^6+T^5-T^4+T^3-T^2+T-1 \cdots \\
\Delta_{8\phantom{'}}(T) = T^8-T^7+T^6-T^5+T^4 -T^3+T^2-T+1 \cdots \\
\Delta_{8'}(T) = T^8-T^7+T^5-T^4+T^2-T+1 \cdots \\
\Delta_{9'}(T) = T^9-T^8+T^5-T^4+T^3-T^2+1 \cdots \\
\end{array}\]\caption{The Alexander polynomials $\Delta_i(T)$.}\label{tablealexpoly}\end{table}

{$\bf|H_1(Y)|=4$:} $Y=\pm\left(-1; \frac{1}{2}, \frac{1}{2}, \frac{1}{n}\right)$ with odd $n$. Either $p/q=4/3$ with $g(K)=1$ and $K$ the trefoil (but this $Y$ is not elliptic),\footnote{It was previously known that there is no finite $4/3$-surgery: finite surgery on a hyperbolic knot must be integral or half-integral~\cite[Theorems~1.1,~1.2]{boyerzhangii}, and no $\pm4/3$-surgery on a torus or satellite knot gives this $Y$~\cite[Theorem~7]{bleilerhodgson2}.}  or $p/q=4$ with $g(K)=1$ or $2$ and Alexander polynomial $\Delta_{1} (T),\Delta_{2} (T),$ or $\Delta_{2'} (T)$. In the latter case,
\[d(S^3_4(K),i) 
= \left\{\begin{array}{cccl}
\phantom{-}\Delta_1 &\phantom{-} \Delta_{2} & \phantom{-}\Delta_{2'} & \\
-5/4 & -5/4 & -13/4 & i = 0\\
\phantom{-}0 & -2 & -2 & i = \pm1\\
-1/4 & -1/4 & -1/4 & i = 2\end{array}\right.\]
These terms match the $d(Y,\spincstr)$ calculated in Example~\ref{ex4} only for $\Delta_1(T)$ and $n=3$, which is $S^3_4(T_{3,2}) = \left( -1; \frac{1}{2}, \frac{1}{2}, \frac{1}{3} \right).$

{$\bf|H_1(Y)|=5~or~6$:} All elliptic $Y$ with these first homologies are lens spaces.

{$\bf|H_1(Y)|=7$:}  $Y=\pm\left(-1; \frac{1}{2}, \frac{1}{3}, \frac{2}{5}\right)$ with $d(Y,\spincstr)$ as in Table~\ref{tableSFScorrterms10}. By Boyer-Zhang, elliptic hyperbolic surgeries must be integral or half-integral \cite[Theorem 1.2]{boyerzhangii}, so it may be $p/q=7/2$ with $g(K)\leq 2$ and Alexander polynomial $\Delta_{1} (T),\Delta_{2} (T),$ or $\Delta_{2'} (T)$, or $p/q=7$ with $g(K)\leq 4$ and appropriate Alexander polynomial. Of the 18 possible sets of $d(S^3_{7/q}(K),\spincstr)$, the only ones that match $d(Y,\spincstr)$ are $7/2$-surgery and $\Delta_1(T)$ (which must be the trefoil) and $7/1$-surgery and $\Delta_{2}(T)$ (which could be $T_{5,2}$). 
Finally, $S^3_{7}(T_{5,2})=S^3_{7/2}(T_{3,2})=\left(-1; \frac{1}{2}, \frac{1}{3}, \frac{2}{5}\right)$.

{$\bf|H_1(Y)|=8$:} $Y= \pm\left(-1; \frac{1}{2}, \frac{1}{2}, \frac{2}{n}\right)$ with $n$ odd. The $d(Y,\spincstr)$ are listed in Table~\ref{tableSFScorrterms}. If $p/q=\pm8$, $g(K)\leq 4$. There are only two possible choices of $\Delta_K(T)$ and $n$ that give the same correction terms; for these cases, $d(S^3_8(K),\spincstr)$, $\Delta_i(T)$, and $n$ are listed in Table~\ref{tablesurgcorrterms}. 

{$\bf|H_1(Y)|=9$:} $Y= \left(-1; \frac{1}{2}, \frac{1}{3}, \frac{2}{3}\right)$ and the correction terms are given in Table~\ref{tableSFScorrterms10}. If $p/q=9$, then $g(K)\leq 5$; if $p/q=9/2$, then $g(K)\leq 2$. Comparing the correction terms shows that the Alexander polynomial must be $\Delta_1(T)$ when $q=1$, so $K$ is the trefoil. 
\end{proof}

\begin{sidewaystable}\centering
\[\begin{array}{cc rrrrrrrrr ccccccccccc }
p & k &  & & &&&&& & d & & \\ 
\hline
4 & 1 & &&&&&&& 0 & \frac{2-n}{4}^*& \frac{-2-n}{4}^*  \\
8 & 1  &&&&&&& \frac{1}{4} & -\frac{1}{4} & \frac{4-n}{8} & \frac{-4-n}{8}  \\
12  & 1 &&&&&& \frac{1}{2} & -\frac{1}{6} & -\frac{1}{6} & \frac{8-n}{12} & \frac{-4-n}{12}   & \frac{4-n}{12}^* & \frac{-8-n}{12}^* \\
16  & 1  &&&&& 0 & 0 & \frac{3}{4} & -\frac{1}{4} & \frac{14-n}{16} & \frac{-2-n}{16}  & \frac{6-n}{16} & \frac{-10-n}{16} \\
20  & 1 & &&& 1 &- \frac{1}{5} &- \frac{1}{5} & \frac{1}{5} & \frac{1}{5}& \frac{2-n}{20} & \frac{22-n}{20}  & \frac{10-n}{20} & \frac{-10-n}{20} &\frac{6-n}{20}^* & \frac{-14-n}{20}^* \\
20  & 2 &&&& 0 & -\frac{2}{5} & - \frac{2}{5}  & \frac{2}{5} & \frac{2}{5}   & \frac{10-n}{20} & \frac{-10-n}{20} & \frac{14-n}{20} & \frac{-6-n}{20}& \frac{2-n}{20}^*& \frac{-18-n}{20}^* \\
24  & 1  &&& \frac{5}{4} &- \frac{1}{4} & -\frac{1}{12} & -\frac{1}{12} &\frac{5}{12}&\frac{5}{12}&\frac{16-n}{24}& \frac{-16-n}{24} & \frac{8-n}{24} & \frac{-8-n}{24} & \frac{8-n}{24} & \frac{32-n}{24}  \\
28  & 1 && \frac{3}{2}  & -\frac{3}{14} & -\frac{3}{14}&  \frac{1}{14} & \frac{1}{14} &\frac{9}{14} & \frac{9}{14}  & \frac{44-n}{28} & \frac{24-n}{28}  & \frac{16-n}{28}  & \frac{-16-n}{28}  & \frac{12-n}{28} & \frac{-4-n}{28}& \frac{8-n}{28}^*& \frac{-20-n}{28}^*\\
28  & 2 &&  -\frac{1}{2}&-\frac{3}{14}& -\frac{3}{14} & \frac{1}{14} & \frac{1}{14}& \frac{9}{14} & \frac{9}{14}   & \frac{4-n}{28} & \frac{-4-n}{28}  & \frac{20-n}{28} & \frac{-8-n}{28} & \frac{24-n}{28} & \frac{-24-n}{28} & \frac{-12-n}{28}^*  & \frac{16-n}{28}^*\\
28  & 3 & &  \frac{1}{2} &\frac{3}{14} & \frac{3}{14}  & -\frac{1}{14}  & -\frac{1}{14} & -\frac{9}{14} & -\frac{9}{14} & \frac{8-n}{28} & \frac{-8-n}{28}  & \frac{20-n}{28} & \frac{-20-n}{28} & \frac{16-n}{28} & \frac{-12-n}{28} & \frac{-4-n}{28}^*& \frac{-32-n}{28}^*\\
32  & 1 &-\frac{7}{4}&  \frac{1}{4} & -\frac{1}{4}& -\frac{1}{4}  &\frac{7}{8} & \frac{7}{8}& -\frac{1}{8}&  -\frac{1}{8}& \frac{26-n}{32} & \frac{58-n}{32} & \frac{2-n}{32} & \frac{34-n}{32} & \frac{18-n}{32} & \frac{-14-n}{32} & \frac{-22-n}{32} & \frac{10-n}{32} \\
32  & 3 & \frac{1}{4} & \frac{1}{4}  & -\frac{1}{4} & -\frac{1}{4} & \frac{5}{8} & \frac{5}{8} &-\frac{3}{8} &-\frac{3}{8} &  \frac{22-n}{32} & \frac{-10-n}{32} & \frac{30-n}{32} & \frac{-2-n}{32} & \frac{6-n}{32} & \frac{-26-n}{32}& \frac{-18-n}{32} & \frac{14-n}{32}
\end{array}\]
\caption{
The correction terms of $Y=\left(-1;\frac{1}{2},\frac{1}{2},\frac{m}{n}\right)$ with $|H_1(Y)|=p\leq 32$ where $p=4m$ and $k=n\bmod m$. Note that $d(Y_{m,-n},\spincstr)=-d(Y_{m,n},-\spincstr)=-d(Y_{m,n},\spincstr)$ since $\left(-1;\frac{1}{2},\frac{1}{2},-\frac{m}{n}\right)=-\left(-1;\frac{1}{2},\frac{1}{2},\frac{m}{n}\right)$. The correction terms come in two flavors, those constant for the entire family, which are of the form $\frac{\cdot}{4m}$ (not necessarily reduced), and those dependent on $n$, which are of the form $\frac{n+\cdot}{4m}$. The correction terms marked by $^*$ correspond to a unique $\spincstr\in \spinc(Y)$ (actually, to the $\spincstr\in Spin(Y)$); the others correspond to two $\spincstr$ each.\label{tableSFScorrterms}}\end{sidewaystable}

\begin{sidewaystable}\centering
\[\begin{array}{rrl | rrrrrrrrrrrrrrrrrr}
 p &  n & \Delta_{i} &d & & & & & & & & \\
\hline
 4&{-3}&\Delta_{1} &-\frac{5}{4}^*&0&-\frac{1}{4}^* & & & & & & \\
 8&{ 3}&\Delta_{1} &-\frac{1}{4}^{\phantom{*}}&\frac{7}{8}&\frac{1}{4}&-\frac{1}{8}& & & & & \\
8&{ -5}&\Delta_{2} & -\frac{1}{4}^{\phantom{*}}&-\frac{9}{8}&\frac{1}{4}&-\frac{1}{8}& & & & & \\
 12&{ 5}&\Delta_{2} &\frac{3}{4}^*&-\frac{1}{6}&\frac{13}{12}&\frac{1}{2}&\frac{1}{12}&-\frac{1}{6}&-\frac{1}{4}^*  & & \\
12&{ -7}&\Delta_{3} & -\frac{5}{4}^*&-\frac{1}{6}&-\frac{11}{12}&\frac{1}{2}&\frac{1}{12}&-\frac{1}{6}&-\frac{1}{4}^*  & & \\
16&{ 7}&\Delta_{3} & -\frac{1}{4}^{\phantom{*}}&\frac{13}{16}&0&\frac{21}{16}&\frac{3}{4}&\frac{5}{16}&0&-\frac{3}{16} & \\
 16&{ -9}&\Delta_{4} &-\frac{1}{4}^{\phantom{*}}&-\frac{19}{16}&0&-\frac{11}{16}&\frac{3}{4}&\frac{5}{16}&0&-\frac{3}{16} & \\
20&{ 9}&\Delta_{4} & \frac{3}{4}^*&-\frac{1}{5}&\frac{19}{20}&\frac{1}{5}&\frac{31}{20}&1&\frac{11}{20}&\frac{1}{5}&-\frac{1}{20} & -\frac{1}{5}& -\frac{1}{4}^* & & & & & & & \\
 20&{ -11}&\Delta_{5} &-\frac{5}{4}^*&-\frac{1}{5}&-\frac{21}{20}&\frac{1}{5}&-\frac{9}{20}&1&\frac{11}{20}&\frac{1}{5}&-\frac{1}{20} & -\frac{1}{5}&-\frac{1}{4}^*  & & & & & & &\\
 24&{11}&\Delta_{5} &-\frac{1}{4}^{\phantom{*}}&\frac{19}{24} &-\frac{1}{12}&\frac{9}{8}&\frac{5}{12}&\frac{43}{24}&\frac{5}{4}&\frac{19}{24}&\frac{5}{12} &\frac{1}{8}& -\frac{1}{12} &-\frac{5}{24}& & & & & & \\
 24&{-13}&\Delta_{6} &-\frac{1}{4}^{\phantom{*}}&-\frac{29}{24}&-\frac{1}{12}&-\frac{7}{8}&\frac{5}{12}&-\frac{5}{24}&\frac{5}{4}&\frac{19}{24}&\frac{5}{12} &\frac{1}{8}& -\frac{1}{12} &-\frac{5}{24}& & & & & & \\
 28&{ -5}&\Delta_{8'} &\frac{ 3}{4}^*&-\frac{3}{14}&\frac{25}{28}&\frac{1}{14}&-\frac{19}{28}&\frac{9}{14}&\frac{1}{28}&-\frac{1}{2}&\frac{29}{28} & \frac{9}{14}& \frac{9}{28} &\frac{1}{14}&-\frac{3}{28}&-\frac{3}{14}&-\frac{1}{4}^*  & & & \\
 28&{-11}&\Delta_{9'} &\frac{3}{4}^*&-\frac{3}{14}&-\frac{31}{28}&\frac{1}{14}&-\frac{19}{28}&\frac{9}{14}&\frac{1}{28}&-\frac{1}{2}&-\frac{27}{28} &\frac{9}{14}& \frac{9}{28} &\frac{1}{14}&-\frac{3}{28}&-\frac{3}{14}&-\frac{1}{4}^* & & & \\
28&{ 13}&\Delta_{6} & \frac{ 3}{4}^*&-\frac{3}{14}&\frac{25}{28}&\frac{1}{14}&\frac{37}{28}&\frac{9}{14}&\frac{57}{28}&\frac{3}{2}&\frac{29}{28} &\frac{9}{14}& \frac{9}{28} &\frac{1}{14}&-\frac{3}{28}&-\frac{3}{14}&-\frac{1}{4} ^*& & & \\
 28&{-15}&\Delta_{7} &-\frac{5}{4}^*&-\frac{3}{14}&-\frac{31}{28}&\frac{1}{14}&-\frac{19}{28}&\frac{9}{14}&\frac{1}{28}&\frac{3}{2}&\frac{29}{28} &\frac{9}{14}& \frac{9}{28} &\frac{1}{14}&-\frac{3}{28}&-\frac{3}{14}&-\frac{1}{4}^* & & & \\
 32&{15}&\Delta_{7} &-\frac{1}{4}^{\phantom{*}}&\frac{25}{32}&-\frac{1}{8}&\frac{33}{32}&\frac{1}{4}&\frac{49}{32}&\frac{7}{8}&\frac{73}{32}&\frac{7}{4} &\frac{41}{32}& \frac{7}{8} &\frac{17}{32}&\frac{1}{4}&\frac{1}{32}&-\frac{1}{8}&-\frac{7}{32}& & \\
 32&{-17}&\Delta_{8} &-\frac{1}{4}^{\phantom{*}}&-\frac{39}{32}&-\frac{1}{8}&-\frac{31}{32}&\frac{1}{4}&-\frac{15}{32}&\frac{7}{8}&\frac{9}{32}&\frac{7}{4} &\frac{41}{32}& \frac{7}{8} &\frac{17}{32}&\frac{1}{4}&\frac{1}{32}&-\frac{1}{8}&-\frac{7}{32}& &\\
\end{array}\]
\caption{All possible cases for $p\leq 32$ where $Y=\left(-1;\frac{1}{2},\frac{1}{2},\frac{m}{n}\right)$ has the same correction terms as some $S^3_p(K)$, assuming the latter exists. Also listed is $n$ and the Alexander polynomial $\Delta_i(T)$ from Table~\ref{tablealexpoly}. Note that $p=4m=|H_1(Y)|$. The correction terms marked by $^*$ correspond to a unique $\spincstr\in \spinc(Y)$ (actually, to the $\spincstr\in Spin(Y)$); the others occur for two $\spincstr$ each.}\label{tablesurgcorrterms}
\end{sidewaystable}

\begin{proof}[Proof of Theorem~\ref{thmz32}] The calculations were performed by computer and are similar to the calculations for $p=8$. We summarize the results:

For each choice of $|H_1(Y)|=4m$, Theorem~\ref{thmh1order} gives a description like $(-1; \frac{1}{2},\frac{1}{2},\frac{m}{n})$ for all possible $Y$. The correction terms for each such manifold are listed in Table~\ref{tableSFScorrterms}. On the other hand, assuming $Y$ is a knot surgery, Theorem~\ref{thmsurgcoeffform} restricts the surgery coefficients that can give $Y$, and Theorems~\ref{thmalex} and~\ref{thmfibered} restrict the Alexander polynomial of the knot (of which there are slightly less than $2^x$ for $x=\frac{p}{2q}$). Using Theorem~\ref{thmd}, it is possible to calculate the correction terms of the resulting surgeries for each knot (assuming they are indeed L-spaces). Table~\ref{tablesurgcorrterms} lists the cases where the correction terms for surgeries on the knots with the given Alexander polynomials from Table~\ref{tablealexpoly} match the correction terms for the appropriate elliptic manifolds.

For uniqueness of the first set of cases, note that these manifolds do not appear in the list obtainable by $p/1$ surgery and so are not surgery on hyperbolic $K$. They are also not on the list of finite satellite surgeries from \cite{bleilerhodgson2} and are not obtainable from surgery on any other torus knots by Corollary~\ref{thmmoser}. 
\end{proof}

\section{Conjectures}\label{sectconj}

We have applied the correction term to obstruct a manifold being surgery on a knot, and it was a sufficient obstruction in all but one case studied, where it was inconclusive; that manifold is realized by a non-integral surgery on a torus knot. On the basis of this evidence, we feel compelled to state the following conjecture, although we do not have any deeper intuition about why it would be true.

\begin{conj}
The Heegaard Floer correction terms $d(Y,\spincstr)$ are sufficient to distinguish which finite manifolds are surgeries on knots in $S^3$.
\end{conj}

A careful examination of Theorem~\ref{thmz32} also suggests a more specific conjecture: all known examples of $\left(-1; \frac{1}{2}, \frac{1}{2}, \frac{m}{n}\right)$ which are knot surgeries obey $n\leq2m+1$ (and the cases $n=2m\pm 1$ are realized by torus knots). Since this paper appeared on the arxiv, the author has proven that each family of dihedral manifolds with a fixed $|H_1(Y)|=4m$ includes finitely many knot surgeries~\cite{doigfiniteii}, but the bounds given there appear to be capable of improvement:

\begin{conj}\label{conjfinite}
If $n>2m+1$, then $\left(-1; \frac{1}{2}, \frac{1}{2}, \frac{m}{n}\right)$ is never a knot surgery.
\end{conj}

Finally, work of Ni and Zhang \cite{nizhangT25} indicates that $7$ and $8$ may be characterizing slopes for $T_{5,2}$ ($p/q$ is a \emph{characterizing slope} for $T_{5,2}$ if $S^3_{p/q}(K) \cong S^3_{p/q}(T_{5,2})$ means $K=T_{5,2}$).

\begin{conj}
The phrase ``With the possible exception of $S^3_7(T_{5,2})$ and $S^3_8(T_{5,2})$'' in Theorem~\ref{thmz2} may be removed. The bound in Corollary~\ref{thmexceptional} may be increased from 7 to 10.
\end{conj}



\bibliographystyle{amsplain}
\bibliography{biblio}

\end{document}